\documentclass{svjour3}
\usepackage[english]{babel}
\usepackage[T1]{fontenc}
\usepackage{lmodern}
\usepackage{amsmath,amssymb,graphicx}
\usepackage{bbm}
\usepackage{cite}
\usepackage{algorithm}
\usepackage{algpseudocode}
\usepackage{subcaption}
\usepackage[utf8]{inputenc}
\newcommand{\dd}{\mathrm{d}}
\newcommand{\expo}{\mathrm{e}}

\spnewtheorem*{lem}{Lemma}{\bf}{\it}
\smartqed

\title{Kinetic equations and self-organized band formations}

\author{Quentin Griette \and Sebastien Motsch}
\institute{Quentin Griette \at Université de Bordeaux - Institut de Mathématiques, 354 cours de la Libération, F 33405, Talence \email{quentin.griette@math.u-bordeaux.fr}
  \and Sebastien Motsch \at Arizona State University - School of Mathematics and Statistical Sciences, 900 S Palm Walk, Tempe, AZ 85281 \email{smotsch@asu.edu}}

\begin{document}

\maketitle


\begin{abstract}
	Self-organization is an ubiquitous phenomenon in nature which can be observed in a variety of different contexts and scales, with examples ranging from fish schools, swarms of birds or locusts, to flocks of bacteria. The observation of such global patterns can often be reproduced in models based on simple interactions between neighboring particles. In this paper we focus on two particular interaction dynamics closely related to the one described in the seminal paper of Vicsek and collaborators. After reviewing the current state of the art in the subject, we study a numerical scheme for the kinetic equation associated with the Vicsek models which has the specificity of reproducing many physical properties of the continuous models, like the preservation of energy and positivity and the diminution of an entropy functional. We describe a stable pattern of bands emerging in the dynamics proposed by Degond-Frouvelle-Liu dynamics and give some insights about their formation. 
\end{abstract}

\section{Introduction}

Swarming dynamics have attracted a lot attention in recent years raising the question how simple interaction rules could lead to complex pattern formation \cite{vicsek_collective_2012}. One of the main difficulty is to link the individual behaviors of agents and the pattern formations observed at a larger scale. Fortunately the framework of kinetic equations allows such transition between microscopic and macroscopic dynamics. Among the many existing swarming models \cite{camazine_self-organization_2001}, the Vicsek model \cite{vicsek_novel_1995} is one of the most popular since it describes a rather simple dynamics (there is only alignment) with few parameters but it is however able to generate complex pattern which are challenging to predict analytically. 
The Vicsek model has been well studied both numerically \cite{nagy_new_2007,chate_modeling_2008,gregoire_onset_2004} or analytically \cite{gamba_global_2016,bolley_mean-field_2012,figalli_global_2018} and the derivation of its kinetic and macroscopic equation is well-understood \cite{degond_hydrodynamic_2013,degond_continuum_2008}. However, as noted first by Chaté and Grégoire\cite{gregoire_onset_2004}, there exists a certain regime where the Vicsek model leads to the formation of traveling bands. Many numerical studies have been conducted to better analyze the formation of these bands at the particle level but none have been proposed so far to study the bands using a kinetic or macroscopic framework. This manuscript aims at proposing a first study on such band formation from the angle of kinetic equations.

After the discovery of band formation in the Vicsek dynamics by Chaté and Grégoire \cite{gregoire_onset_2004}, there has been a debate \cite{aldana_phase_2003,chate_comment_2007} about the order of the phase transition in the Vicsek model (continuous or discontinuous). As there was no analytic framework available, the conjecture could be only based on (particle) numerical simulations. However, the derivation of kinetic and macroscopic equation for the Vicsek model \cite{degond_continuum_2008,degond_hydrodynamic_2013} indicated that in a dense regime of particles (the so-called moderately interacting particle \cite{oelschlager_law_1985}), the Vicsek model has a continuous transition from order to disorder. In this regime of high density, no phase transition or band formation could be observed. A major discovery was then provided by Degond, Frouvelle and Liu \cite{frouvelle_dynamics_2012,degond_macroscopic_2013,degond_phase_2015} where a modification of the (continuous) Vicsek was considered: alignment is proportional to the local density rather than to the mean direction. In their dynamics, a phase transition occurs: at low density, the velocity distribution becomes uniform, whereas at large density, the dynamics converge to a so-called von Mises distribution. This analytic result was only proven in an homogeneous setting (no spatial variable). Thus, it is still unknown what effect a transport term would have on the dynamics. This is however a very  challenging question as the transport term breaks the entropy dissipation. In this manuscript, we propose to investigate numerically the Degond-Frouvelle-Liu (DFL) dynamics in a non-homogeneous setting.

Starting from the kinetic equation associated with the Vicsek model, we first review some properties of the collisional operator (entropy dissipation) that will be central for the building of our numerical scheme. Most of these estimates are built on the Fokker-Planck structure of the operator. We do take advantage of this formulation in the design of our numerical scheme. The key properties of the collision operator (positivity preserving, entropy dissipation) are also satisfied by the discrete operator. Since we aim at analyzing the long-time behavior of the solution, it is essential to preserve these properties. For instance, several papers have already proposed to solve numerically the kinetic equation associated with the Vicsek model using other methods (spectral method \cite{gamba_spectral_2015}, particle method \cite{dimarco_self-alignment_2016}, discontinuous Galerkin \cite{filbet_discontinuous-galerkin_2017}). But we would rather have lower accuracy and a preserving numerical scheme to study the long-time behavior of the solution (even though our scheme is still second order accurate in the velocity-variable). We then explore the dynamics of the kinetic equation in various regimes. In the original Vicsek model, no band formations are observed, the spatial density becomes homogeneous while the velocity distribution becomes distributed according to a (global) von Mises distribution. In the Degond-Frouvelle-Liu (DFL) dynamics, however, when the density is above a threshold, band formation occurs starting from random initial configuration. As far as the authors know, this is the first time such band formations are observed at the kinetic level. Band were also observed in \cite{filbet_discontinuous-galerkin_2017} but there were only 'transient', the density profile would become flat after a long time. Here, the density profile is not  flattening out, but instead is becoming more and more concentrated. Numerically, we have to introduce an adaptive time step to deal with a demanding CFL condition because of this very phenomenon.

Although our numerical investigation suggests that band formation emerges from the DFL dynamics, it would be crucial to also develop an analytic framework to further understand this phenomenon. Our results indicate that the transport operator could further amplify the concentration originating from the alignment operator. From these observations, it seems unlikely that there exists an analytic profile for these bands. But the question remains open. Similarly, we could perform simulation in dimension 3, but the discretization of the unit sphere $\mathbb{S}^2$ is more delicate than  $\mathbb{S}^1$ (there is no 'uniform grid' on $\mathbb{S}^2$) and thus having discrete entropy dissipation or symmetry preserving would be more challenging. Finally, higher order accuracy in time should also be investigated using for instance \cite{gottlieb_strong_2001,carrillo_finite-volume_2015}.

\section{Microscopic description}

\subsection{Vicsek model}

The Vicsek model \cite{vicsek_novel_1995,degond_continuum_2008} at the particle level describes the motion of $N$ particles with position ${\bf x}_i\in\mathbb{R}^d$ (with $d=2,3$) and a direction $\omega_i\in\mathbb{S}^{d-1}$ (i.e. $|\omega_i|=1$). The evolution of the particles is given by the following system:
\begin{equation}
	\label{eq:micro}
	\begin{array}{rcl}
		{\bf x}_i' &=& c\, \omega_i \\
		\dd \omega_i &=& P_{\omega_i^\perp}(\mu\,\Omega_i\dd t + \sqrt{2\sigma}\circ\dd B_i^t),
	\end{array}
\end{equation}
where $c>0$ is the speed of the particle, $\mu$ is the strength of the alignment interaction, $\sigma$ is the intensity of the noise,  $\dd B_i^t$ are independent white noises, $P_{\omega_i^\perp}$ is the orthogonal projection on the orthogonal of $\omega_i$ defined as:
\begin{equation}
	\label{eq:proj_orth}
	P_{\omega_i^\perp}=\mbox{Id}-\omega_i\otimes\omega_i,
\end{equation}
which ensures that $|\omega_i(t)|=1$ over time, and $\Omega_i$ is the average direction of the particle $i$:
\begin{equation*}
	\label{eq:bigOmega}
	\Omega_i = \frac{{\bf j}_i}{|{\bf j}_i|} \;,\quad {\bf j}_i = \sum_{j,|{\bf x}_j-{\bf x}_i|\leq R} \omega_j,
\end{equation*}
where  $R$ is the radius of interaction.

\subsection{Degond-Frouvelle-Liu (DFL) dynamics}

Degond, Frouvelle and Liu \cite{frouvelle_dynamics_2012,degond_phase_2015} proposed a modification of the dynamics where the alignment interaction $\mu$ is proportional to the norm of the flux ${\bf j}_i$:
\begin{equation}
	\label{eq:micro_amic}
	\begin{array}{rcl}
		{\bf x}_i' &=& c \omega_i \\
		\dd \omega_i &=& P_{\omega_i^\perp}(\mu\,{\bf j}_i\dd t + \sqrt{2\sigma}\circ\dd B_i^t).
	\end{array}
\end{equation}
This modification has several consequences: i) unlike the Vicsek dynamics \eqref{eq:micro} which is ill-defined when the flux ${\bf j}_i$ equal zero ($\Omega_i$ not defined),  the DFL dynamics does not have any singularity, ii) there is a phase transition in the dynamics \eqref{eq:micro_amic} as the number of particles increases (or similarly as $\mu$ increases). The kinetic description of this dynamics will allow to better explain this phase transition (see section \ref{sub:homog_case}).

\subsection{Band formation}

Band formations have been first analyzed by Grégoire and Chaté  \cite{gregoire_onset_2004} in the case of the original {\it discrete} Vicsek model and several numerical studies have been conducted since \cite{chate_modeling_2008,nagy_new_2007,aldana_phase_2003}. To motivate our study, we present numerically an example of such band formation in the context of the {\it continuous} dynamics \eqref{eq:micro}.

The numerical simulation presented in this subsection is performed with $N=30,000$ particles on a square domain with length $L=4$ and periodic boundary condition. Initially, particles are distributed at random in space and velocity. Table \ref{tab:param_micro} gives the list of values for the parameters. We observe in Figure \ref{fig:simu_micro} the formation of a traveling wave moving in the $x$-direction. To further quantify this formation, we estimate the average density $\rho$ and velocity $u$ in the $x$-direction:
\begin{eqnarray*}
	\label{eq:rho_in_x}
	\overline{\rho}(x,t) \Delta x &=& \sum_{i=1}^N \mathbbm{1}_{[-\frac{\Delta x}{2},\frac{\Delta x}{2})}(x_i(t)-x) \\
	\overline{\rho}(x,t)\overline{u}(x,t) &=& \sum_{i=1}^N \mathbbm{1}_{[-\frac{\Delta x}{2},\frac{\Delta x}{2})}(x_i(t)-x) \cos \theta_i(t).
\end{eqnarray*}
where $x_i$ and $\cos \theta_i$ are respectively
the $x$-component of the position vector ${\bf x}_i$ and velocity $\omega_i$. We give an example of such $\rho$ and $u$ in Figure \ref{fig:simu_micro_average}.

We notice that the regime in which the band formation occurs is far from being dense. Indeed, in a homogeneous setting, the average number of neighbors is given by:
\begin{displaymath}
	\text{Average number of neighbors} \approx \frac{|B(0,R)|}{L^2}\times N =2.36,
\end{displaymath}
therefore we are far from being in kinetic regime (let alone macroscopic region). Thus, the validity of the kinetic equation associated with the dynamics (described in the next section) is questionable in this regime. Particles are not necessarily 'moderately interacting' \cite{oelschlager_law_1985}. 

\begin{figure}[ht]
	\centering
	\includegraphics[width=.49\textwidth]{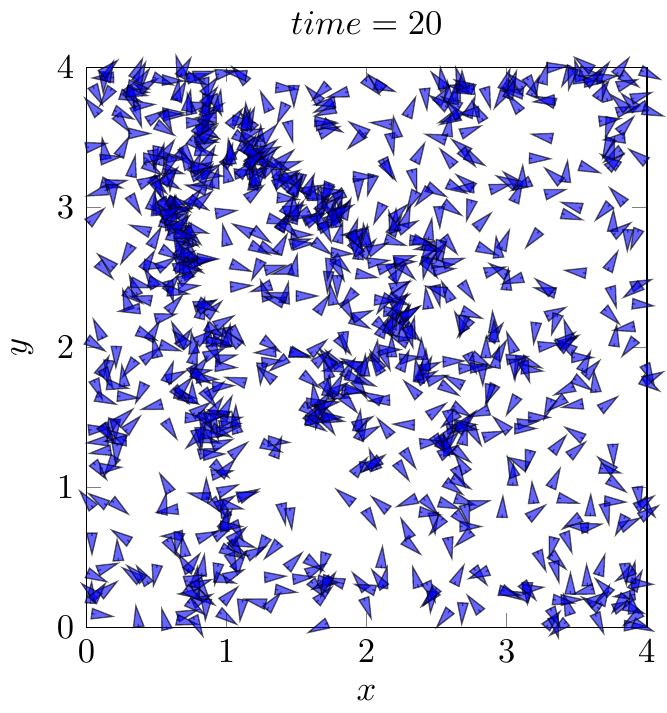}
	\includegraphics[width=.49\textwidth]{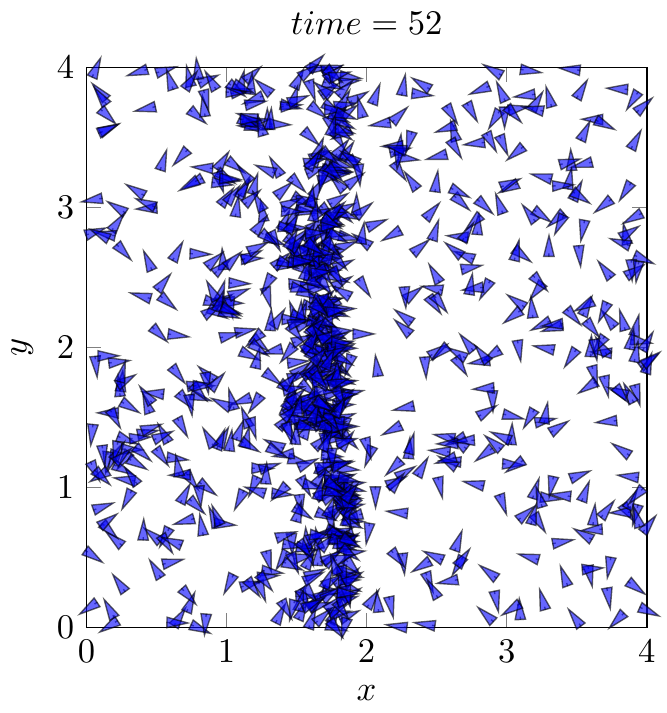}
	\caption{Illustration of the simulation of the Vicsek model \eqref{eq:micro} at two different time. We observe the formation of a vertical {\it band}. See table \ref{tab:param_micro} for the parameters used.}
	\label{fig:simu_micro}
\end{figure}

\begin{figure}[ht]
	\centering
	\includegraphics[width=.49\textwidth]{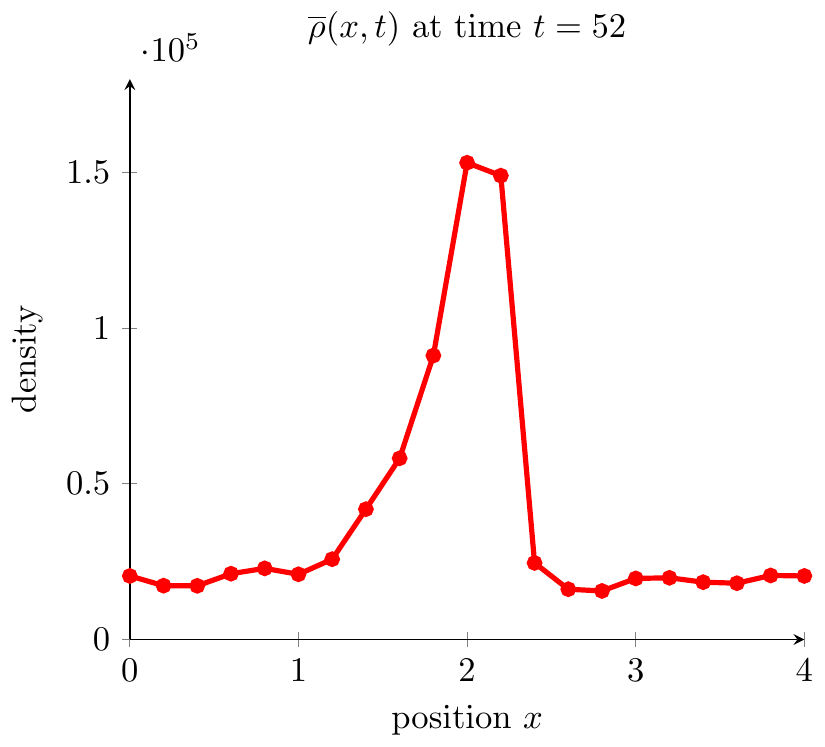}
	\includegraphics[width=.49\textwidth]{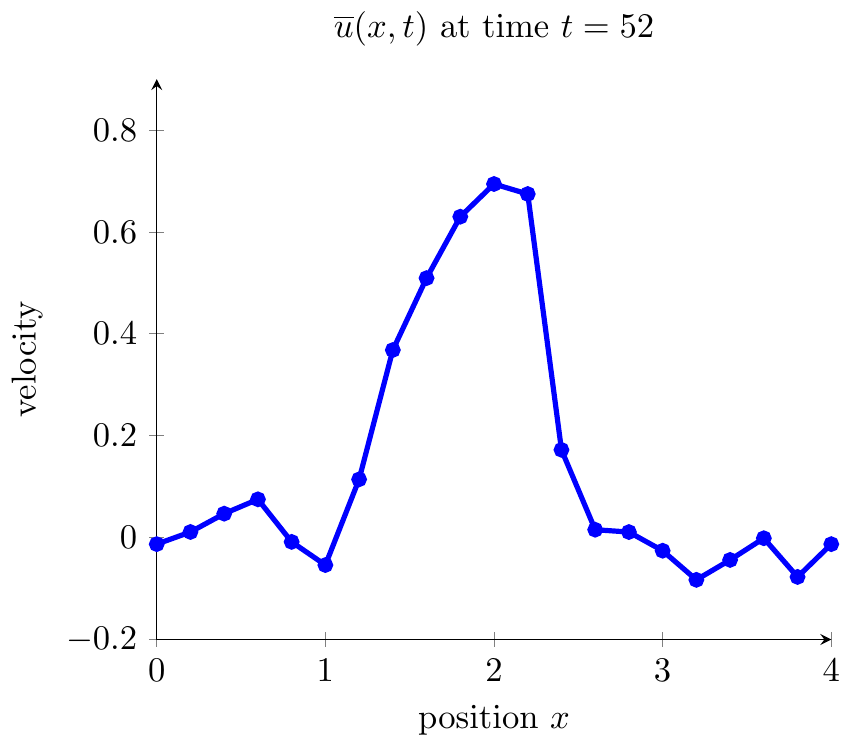}
	\caption{Density $\rho$ and average velocity $u$ in the $x$-direction at $t=52$. Where the density $\rho$ is larger, the speed $u$ increases.}
	\label{fig:simu_micro_average}
\end{figure}

\begin{table}[ht]
	\centering
	\begin{tabular}{l|c|c}
		Description & notation & value \\
		\hline
		Number particles & $N$ & $30,000$\\
		Strength alignment & $\mu$ & $100$\\
		Noise intensity & $\sigma$ & $20$\\
		Radius interaction & $R$ & $.02$\\
		Length domain & $L$ & $4 $\\
		Time step & $\Delta$t & $10^{-2}$
	\end{tabular}
	\caption{Parameters used in the simulations for Figures \ref{fig:simu_micro}-\ref{fig:simu_micro_average}}
	\label{tab:param_micro}
\end{table}

\section{Kinetic description}

\subsection{Introduction}

The kinetic equation associated with the Vicsek dynamics \eqref{eq:micro}  is described through the density distribution $f({\bf x},\omega,t)$. As the number of particles $N$ tends to infinity, the particle dynamics converge to the solution of a deterministic equation given by \cite{bolley_mean-field_2012} 
\begin{equation}
	\label{eq:kinetic_Vicsek}
	\partial_t f + c\, \omega\cdot\nabla_{\bf x} f= -\mu\nabla_\omega\cdot (P_{\omega^\perp}(\Omega_f)\,f) + \sigma\Delta_\omega f,
\end{equation}
where $c>0$ is the speed of the particles, $\mu>0$ is the intensity of the relaxation toward the mean velocity,  $\sigma>0$ is the diffusion coefficient, $P_{\omega^\perp}$ the projection operator \eqref{eq:proj_orth}, $\Omega$ is the mean velocity at the point ${\bf x}$
\begin{equation*}
	\label{eq:bar_omega}
	\Omega_f({\bf x},t) = \frac{{\bf j}_f({\bf x},t)}{|{\bf j}_f({\bf x},t)|} \quad \text{with} \quad {\bf j}_f({\bf x}, t) = \int_{y\in B({\bf x},R),\omega\in\mathbb{S}^{d-1}}\omega f({\bf y},\omega, t)\,\dd {\bf y} \dd \omega,
\end{equation*}
$R>0$ being the radius of interaction.

The DFL dynamics lead to a similar kinetic equation except for the transport term in $\omega$:
\begin{equation}
	\label{eq:kinetic_Amic}
	\partial_t f + c\, \omega\cdot\nabla_{\bf x} f= -\mu\nabla_\omega\cdot (P_{\omega^\perp}({\bf j}_f)\,f) + \sigma\Delta_\omega f.
\end{equation}
in other words the strength $\mu$ of alignment is now proportional to $|{\bf j}_f|$. 

\subsection{Homogeneous case}
\label{sub:homog_case}

To investigate kinetic equations, we study the homogeneous case, assuming that $f$ is independent of ${\bf x}$. Thus, the kinetic equations \eqref{eq:kinetic_Vicsek} and \eqref{eq:kinetic_Amic} become:
\begin{equation}
	\label{eq:homog}
	\partial_t f = Q(f),
\end{equation}
with:
\begin{eqnarray}
	\label{eq:Q}
	Q(f) &=& -\mu_f \nabla_\omega\cdot (P_{\omega^\perp}(\Omega_f)\,f) + \sigma\Delta_\omega f,
\end{eqnarray}
and $\Omega_f = \frac{{\bf j}_f}{|{\bf j}_f|}$ with ${\bf j}_f = \int_{\omega\in\mathbb{S}^{d-1}}\omega f(\omega)\,\dd \omega$
and
\begin{equation*}
	\label{eq:mu_both}
	\mu_f = \left\{
		\begin{array}{ll}
			\mu & \text{Vicsek dynamics} \\
			& \\
			\mu |{\bf j}| & \text{DFL dynamics}.
		\end{array}
	\right.
\end{equation*}
The operator $Q$ defined by \eqref{eq:Q} can be written as a Fokker-Planck type operator. Introducing:
\begin{equation}
	\label{eq:phi}
	\phi(\omega) = \left\{
		\begin{array}{ll}
			\langle{\bf j}_f,\omega\rangle & \text{Vicsek dynamics} \\
			&\\
			\langle\Omega_f,\omega\rangle & \text{DFL dynamics},
		\end{array}
	\right.
\end{equation}
with $\langle,\rangle$ the usual scalar product in $\mathbb{R}^n$, we find:
\begin{equation}
	\label{eq:FP}
	Q(f) = \sigma \nabla_\omega\cdot\left(M_f \nabla_\omega\cdot\left(\frac{f}{M_f} \right)\right), \quad \;\text{with} \quad M_f(\omega) = \expo^{\frac{\mu}{\sigma} \phi(\omega)},
\end{equation}
using the identity  $\nabla_\omega \langle{\bf u},\omega\rangle  = P_{\omega^\perp} ({\bf u})$. We deduce a first identity:
\begin{equation}
	\label{eq:decay_L2}
	\int_\omega \partial_t f \frac{f}{M_f} \, \dd \omega =  -\sigma \int_\omega M_f \left|\nabla_\omega \left(\frac{f}{M_f} \right)\right|^2\,\dd \omega \leq 0.
\end{equation}
Unfortunately, the left-hand side of \eqref{eq:decay_L2} cannot be written as a total time derivative and thus we cannot deduce any {\it entropy decay}. The {\it trick} is to notice the following:
\begin{eqnarray*}
	Q(f) &=& \sigma \nabla_\omega\cdot\left(f \frac{M_f}{f} \,\nabla_\omega\cdot\left(\frac{f}{M_f} \right)\right) \\
	&=& \sigma \nabla_\omega\cdot\left(f  \,\nabla_\omega\cdot\ln \left(\frac{f}{M_f} \right)\right).
\end{eqnarray*}
Therefore,
\begin{equation}
	\label{eq:decay_entropy}
	\int_\omega \partial_t f\, \ln \left(\frac{f}{M_f}\right) \, \dd \omega =  -\sigma \int_\omega f \left|\nabla_\omega \ln \left(\frac{f}{M_f} \right)\right|^2\,\dd \omega \leq 0.
\end{equation}
Thanks to the property of the logarithm, the left-hand side can now be written as a total time derivative and we deduce the following proposition.
\begin{proposition}
	Suppose $f$ is a solution to the homogeneous kinetic equation \eqref{eq:homog} and consider the {\it free energy}:
	\begin{equation}
		\label{eq:free_energy}
		\mathcal{F}[f] = \int_\omega f\,\ln f\,\dd \omega - \frac{\mu}{\sigma} \Phi_f,
	\end{equation}
	with:
	\begin{equation}
		\label{eq:Phi}
		\Phi_f = \left\{
			\begin{array}{ll}
				|{\bf j}_f| & \text{Vicsek dynamics} \\
				&\\
				\frac{1}{2} |{\bf j}_f|^2 & \text{DFL dynamics}.
			\end{array}
		\right.
	\end{equation}
	It satisfies:
	\begin{equation*}
		\label{eq:decay_free_energy2}
		\frac{d}{dt} \mathcal{F} = -\sigma \int_\omega f \left|\nabla_\omega \ln \left(\frac{f}{M_f} \right)\right|^2\,\dd \omega \leq 0.
	\end{equation*}
\end{proposition}
\begin{proof}
	We remark that left-hand side of \eqref{eq:decay_entropy} is a total derivative:
	\begin{eqnarray*}
		\int_\omega \partial_t f\, \ln \left(\frac{f}{M_f}\right) \, \dd \omega &=& \int_\omega \partial_t f\, \left(\ln f - \frac{\mu}{\sigma} \phi\right) \, \dd \omega \\
							&=& \int_\omega \partial_t (f\,\ln f) - \frac{\mu}{\sigma} (\partial_t f)\phi \, \dd \omega,
	\end{eqnarray*}
	using the conservation of mass $\int_\omega \partial_t f \, \dd \omega=0$. Notice moreover that $\phi$ \eqref{eq:phi} can be expressed as gradient (making the dynamics \eqref{eq:Q} a gradient flow as noted in \cite{figalli_global_2018}). Indeed, taking $f+{h}$ a small perturbation of $f$, we have:
	\begin{eqnarray*}
		|{\bf j}_{f+{h}}|^2 &=& \left|\int_\omega (f+{h})\omega\,\dd \omega\right|^2 = |{\bf j}_f|^2\; +\; 2 \langle\int_\omega f\omega\,\dd \omega,\int_\omega {h} \omega\,\dd \omega\rangle\; +\; \mathcal{O}({h}^2) \\
		      &=& |{\bf j}_{f}|^2\; +\; 2\int_\omega \langle{\bf j}_f,\omega\rangle {h} \,\dd \omega\;  + \;\mathcal{O}({h}^2),
	\end{eqnarray*}
	thus $\frac{\delta |{\bf j}_f|^2}{\delta f}(\omega) = 2 \langle{\bf j}_f,\omega\rangle$. In particular $\phi = \frac{\delta \Phi}{\delta f}$ with $\Phi$ given by \eqref{eq:Phi}. We deduce:
	\begin{equation*}
		\int_\omega (\partial_t f) \phi(\omega)\,\dd \omega = \int_\omega (\partial_t f) \frac{\delta \Phi}{\delta f}(\omega)\,\dd \omega = \frac{\text{d}}{\text{d}t} \Phi(f(t)).
	\end{equation*}
	Therefore, we obtain:
	\begin{eqnarray*}
		\int_\omega \partial_t f\, \ln \left(\frac{f}{M_f}\right) \, \dd \omega &=& \frac{\text{d}}{\text{d}t} \mathcal{F},
	\end{eqnarray*}
	with $\mathcal{F}$ given by \eqref{eq:free_energy}.
\end{proof}

\subsection{Phase transition}

Since the dynamics \eqref{eq:Q} have entropy, one can study the long-time behavior and deduce the convergence toward an equilibrium given as the minimizer of the {\it free energy} $\mathcal{F}$ \eqref{eq:free_energy}. But first we need to identify the minimizers of $\mathcal{F}$. To do so, we notice that once we fix the flux ${\bf j}_f$, the minimizer would be given by von Mises.
\begin{lemma}\label{lem:affinemin}
	Fix ${\bf j}$ with $0<|{\bf j}|<1$ and consider the affine space:
	\begin{equation*}
		\label{eq:affine}
		\mathcal{A} = \left\{f\in L^2(\mathbb{S}^{n-1}) \;\;| \;\; \int_{\mathbb{S}^{n-1}} \omega f(\omega)\,\dd \omega = {\bf j} \quad \text{and}\quad \int_{\mathbb{S}^{n-1}} f(\omega)\,\dd \omega = 1\right\}.
	\end{equation*}
	Then:
	\begin{equation*}
		\label{eq:min_von}
		\inf_{\mathcal{A}} \left\{\int_\omega f\,\ln f\,\dd \omega\right\} = \int_\omega M_*\,\ln M_*\,\dd \omega,
	\end{equation*}
	where $M_*$ is the von Mises \eqref{eq:FP} satisfying $\int_\omega \omega M_*(\omega)\,\dd \omega = {\bf j}$.
\end{lemma}
\begin{proof}
	Assume there exists a minimizer $f_*$ and rewrite the constraint as
	\begin{equation*}
		\label{eq:rewrite}
		H[f] = \int f \ln f \;\;,\;\;\;\;
		\alpha[f] = \left| \int_{\mathbb{S}^{n-1}}\!\!\!\!\!\! \omega f - {\bf j} \right|^2 - 1 \;\;,\;\;\;\;
		\beta[f] = \left( \int_{\mathbb{S}^{n-1}}\!\!\!\!\! f - 1 \right)^2.
	\end{equation*}
	Denote $\lambda_1$ and $\lambda_2$ the Lagrange multiplier associated with $f_*$:
	\begin{eqnarray*}
		\label{eq:lag}
		\left.\frac{\delta H}{\delta f}\right|_{f_*} &=& \lambda_1 \left.\frac{\delta\alpha}{\delta f}\right|_{f_*} + \lambda_2 \left.\frac{\delta\beta}{\delta f}\right|_{f_*} \\
		\label{eq:lag2}
		\Rightarrow \ln f_* + 1 &=& \lambda_1 \,2\omega\cdot(-{\bf j} ) + \lambda_2 \cdot0.
	\end{eqnarray*}
	since $\int f_* = 1$. Thus, taking the exponential leads to:
	\begin{displaymath}
		f_*(\omega) = C \expo^{-2\lambda_1 \omega\cdot{\bf j}},
	\end{displaymath}
	and therefore $f_*$ is a von Mises distribution.
\end{proof}
As a consequence of Lemma \ref{lem:affinemin}, we can restrict the search of minimizers of the free energy $\mathcal{F}$ on von Mises distributions. In Figure \ref{fig:phase_transition}, we estimate numerically the entropy $\int_\omega M \ln M$ of von Mises distributions depending on their average velocity $|{\bf j}|=|\int_\omega \omega M|$ along with its approximation near ${\bf j}\approx0$:
\begin{equation*}
	\label{eq:approx_entropy}
	\int_\omega M \log M = -\log 2\pi + |{\bf j}|^2 \;+\; \mathcal{O}(|{\bf j}|^3).
\end{equation*}
We deduce that the free energy $\mathcal{F}$ for the Vicsek model will never have a minimum at ${\bf j}=0$ meaning that the uniform distribution is never stable. However, for the DFL dynamics, when the diffusion $\sigma$ is large, the free entropy $\mathcal{F}$ will be minimum at ${\bf j}=0$ and therefore the uniform distribution will become the stable equilibrium. These two situations are depicted in Figure \ref{fig:phase_transition_2}.

\begin{figure}[t]
	\centering
	\includegraphics[scale=.9]{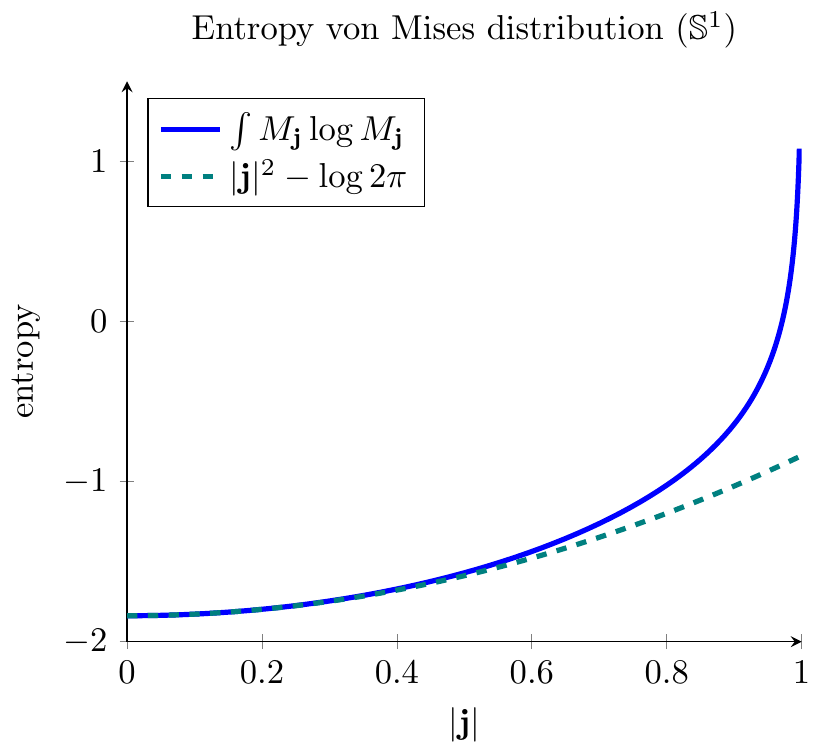}
	\caption{Entropy $\int M \log M$ for $M$ von Mises distribution as a function of the length $|\int\omega M|$. The curve increases quadratically near $|{\bf j}|=0$.}
	\label{fig:phase_transition}
\end{figure}

\begin{figure}[t]
	\centering
	\includegraphics[width=.47\textwidth]{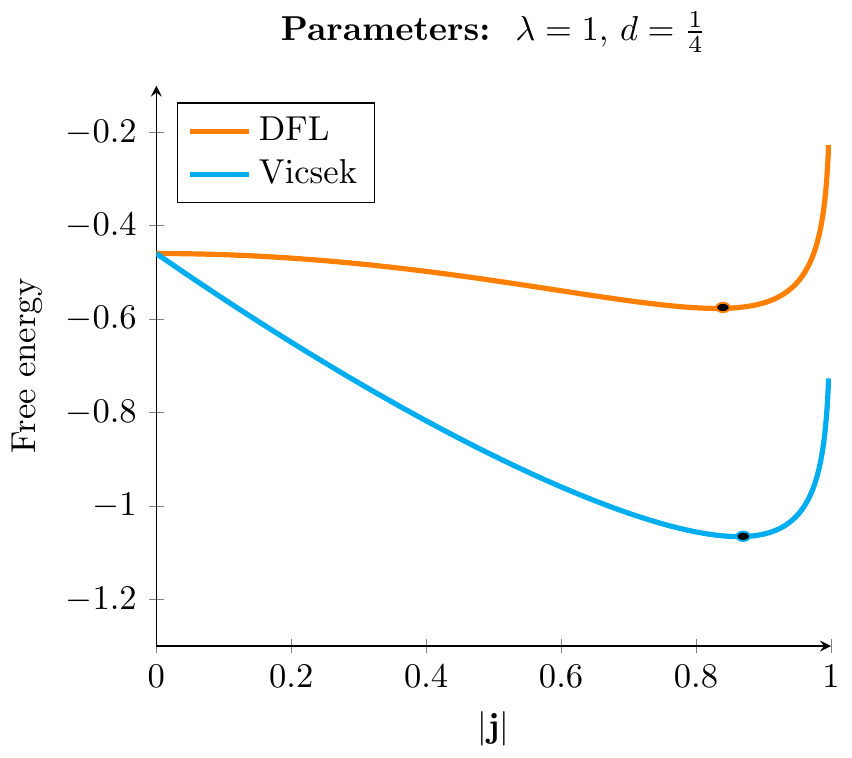} \qquad
	\includegraphics[width=.47\textwidth]{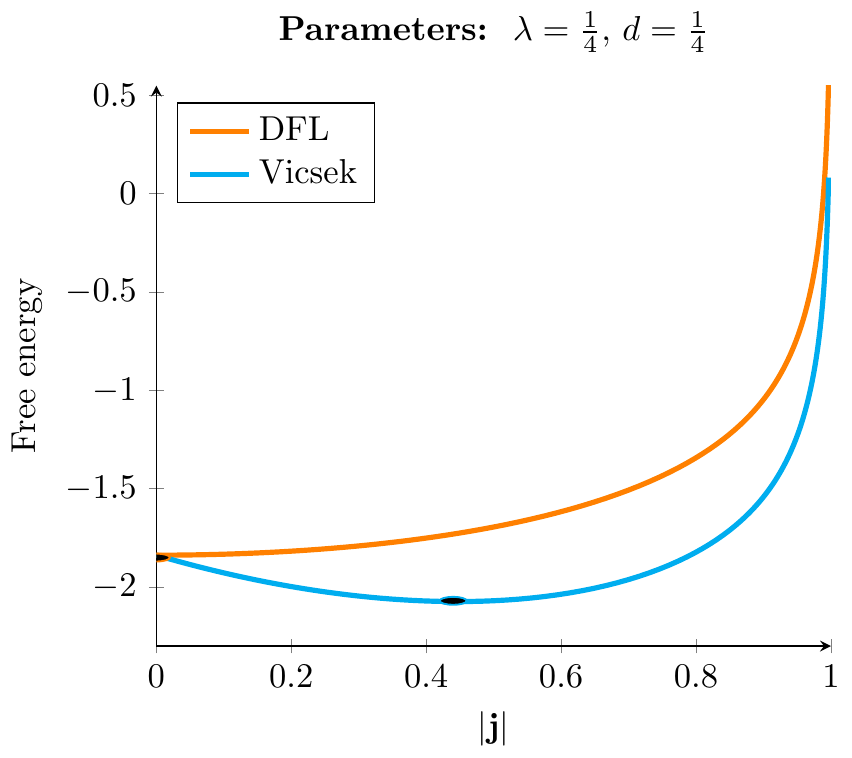}
	\caption{{\bf Left}: for low value of the diffusion coefficient $\sigma$, the minimizers for both free energies are von Mises distribution. {\bf Right}: when the diffusion coefficient $\sigma$ exceeds a certain threshold, the uniform distribution, i.e. ${\bf j}=0$, becomes the minimizer for the DFL dynamics.}
	\label{fig:phase_transition_2}
\end{figure}

\section{Numerical scheme}

Several schemes have already been proposed to study the kinetic equation \eqref{eq:kinetic_Vicsek} using spectral method \cite{gamba_spectral_2015}, discontinuous Galerkin \cite{filbet_discontinuous-galerkin_2017} or semi-Lagrangian \cite{dimarco_self-alignment_2016}. However, we are now interested in the long time behavior of the solution, thus we would like to design a numerical scheme with several properties:
\begin{itemize}
	\item conservative, preserve positivity (under some CFL condition)
	\item satisfy a discrete version of inequalities \eqref{eq:decay_L2} and \eqref{eq:decay_entropy}
\end{itemize}
In the following, we study the 2D scenario, taking advantage of the fact that the velocity space $\omega\in\mathbb{S}^1$ can be parametrized using polar coordinates by $\theta\in\mathbb{R}/2\pi\mathbb Z$ with $\omega = \left(\begin{matrix} \cos\theta \\ \sin\theta\end{matrix}\right)$. Thus, the kinetic equation \eqref{eq:kinetic_Vicsek} becomes:
	\begin{equation*}
		\label{eq:polar}
		\partial_t f + c\,\omega\cdot\nabla_{\bf x}f = -\mu_f\,\partial_\theta(\sin(\bar{\theta} - \theta)f)+\sigma\partial^2_\theta f,
	\end{equation*}
where $\bar{\theta}$ is such that:
\begin{equation*}
	\label{eq:bar_theta}
	\Omega = \left(
	\begin{matrix}
		\cos\bar{\theta} \\
		\sin\bar{\theta}
	\end{matrix}\right),
\end{equation*}
and $\mu_f$ is either a constant (Vicsek model) or proportional to $|{\bf j}|$ (DFL dynamics). 
Our numerical scheme is then based on a splitting method solving separately: 
\begin{itemize}
	\item the {\it transport part}
		\begin{equation}\label{eq:transport}
			\partial_t f + c\,\omega\cdot \nabla_{\bf x} f = 0, 
		\end{equation}
	\item the {\it collision part} :
		\begin{equation}
			\label{eq:collision}
			\partial_t f = Q(f),
		\end{equation}
		where $ Q(f) = -\mu_f \partial_\theta( \sin(\bar{\theta}-\theta)f) + \sigma\partial_\theta^2 f $.

\end{itemize}


\subsection{Collision operator}\label{sec:collision}

In this section we focus on numerically solving equation \eqref{eq:collision}. We write:
\begin{equation*}\label{eq:defcollision}
	Q(f)=-\mu_f\partial_\theta(\sin(\bar{\theta} - \theta)f)+\sigma\partial_\theta^2 f = \sigma\partial_\theta \left(M_{f}\partial_\theta \left(\frac{f}{M_{f}}\right)\right),
\end{equation*}    
where $M_{f}$ is the {\it Von Mises distribution} :
\begin{equation*}
	M_{f}(\theta) = C_0 \exp^{\frac{\mu_f}{\sigma}\cos(\theta-\bar{\theta})},
\end{equation*}
and $C_0$ is a normalization constant. Here, we simply take $C_0=1$.

\subsubsection{Discretization in $ \theta $}

Fix $N>0$ and consider a uniform discretization of the interval $[0,2\pi)$ with $\theta_k=i\Delta\theta$ and $\Delta\theta=\frac{2\pi}{N}$. Denote: $f_k=f(\theta_k)$ and $f_{k+\frac12}=f(\theta_k+\Delta\theta/2)$ and similarly $M_k = M_f(\theta_k)$. To approximate $Q$ (which is a differential operator), we use the second order approximation:
\begin{displaymath}
	\partial_\theta \left. \left(\frac{f}{M_f}\right)\right|_{\theta_k} = \frac{1}{\Delta\theta} \left( \frac{f_{k+\frac12}}{M_{k+\frac12}} - \frac{f_{k-\frac12}}{M_{k-\frac12}}\right) + \mathcal{O}(\Delta\theta^2),
\end{displaymath}
which gives
\begin{equation*}
	Q(f)(\theta_k) = Q_N(f)(\theta_k) +  \mathcal{O}(\Delta\theta^2),
\end{equation*}
with
\begin{eqnarray}
	\label{eq:Q_N}
	Q_N(f)(\theta_k) &=& \frac{\sigma}{\Delta\theta^2}\left[M_{k+\frac12}\left(\frac{f_{k+1}}{M_{k+1}}-\frac{f_{k}}{M_{k}}\right) -
		  M_{k-\frac12}\left(\frac{f_{k}}{M_{k}}-\frac{f_{k-1}}{M_{k-1}}\right) \right] \\ 
		 \notag &=& \frac{\sigma}{\Delta\theta^2}\left[\frac{M_{k+\frac12}}{M_{k+1}}f_{k+1}-\left(\frac{M_{k+\frac12}+M_{k-\frac12}}{M_k}\right) f_k + \frac{M_{k-\frac12}}{M_{k-1}}f_{k-1}\right].
\end{eqnarray}
The discrete operator $Q_N$ can be identified with a square $N\times N$ matrix:
\begin{equation*}
	\label{eq:Q_N_matrix}
	Q_N:= \frac{\sigma}{\Delta\theta^2}
	\left(\begin{matrix} 
		b_1     & c_1   & 0     & \cdots    &   0       & a_1   \\
		a_2     & b_2   & c_2   & \cdots    &           & 0     \\
		0       & a_3   & b_3   &           &           & \vdots\\
		\vdots  &       &       & \ddots    &           &        \\    
		0       &       &       &           & b_{n-1}   & c_{n-1}\\
		c_n     & 0     & \cdots&           & a_{n}   & b_n
	\end{matrix}\right),
\end{equation*}
and
\begin{equation*}
	\label{eq:abc}
	a_k = \frac{M_{k-\frac12}}{M_{k-1}},\;\;\; b_k = -\frac{M_{k+\frac12} + M_{k-\frac12}}{M_k},\;\;\; 
	c_k = \frac{M_{k+\frac12}}{M_{k+1}},
\end{equation*}
with a slight abuse of notation such as $M_{-\frac12}=M_{N-\frac12}$ (by the periodicity of $M_f$).


The discrete operator $Q_N$ reproduces many features of the differential operator $Q$. Define the scalar product:
\begin{displaymath}
	\langle u,v\rangle_{M_f^{-1}} = \sum_{i=1}^N \frac{u_k v_k}{M_k},  
\end{displaymath}
the operator $Q_N$ then satisfies the discrete equivalent to  \eqref{eq:decay_L2} and \eqref{eq:decay_entropy}.
\begin{proposition}
	The operator $Q_N$ \eqref{eq:Q_N} is symmetric with respect to this scalar product:
	\begin{displaymath}
		\langle Q_N(u),v\rangle_{M_f^{-1}} = \langle u,Q_N(v)\rangle_{M_f^{-1}},
	\end{displaymath}
	and satisfies:
	\begin{eqnarray}
		\label{eq:decay_L2_discrete}
		\langle Q_N(u),u\rangle_{M_f^{-1}} &=& - \frac{\sigma}{\Delta\theta^2} \sum_k M_{k+\frac12} \left(\frac{u_{k+1}}{M_{k+1}} - \frac{u_{k}}{M_{k}}\right)^2 \leq 0, \\
		\label{eq:decay_entropy_discrete}
		\langle Q_N(u),\ln \frac{u}{M_f} \rangle &\leq& 0.
	\end{eqnarray}
	\begin{displaymath}
	\end{displaymath}
\end{proposition}
\begin{proof}
	Take any vectors $u$ and $v$, then the Abel formula (discrete integration by parts) gives:
	\begin{eqnarray*}
		\label{eq:sym_Q_N}
		\langle Q_N(u),v\rangle_{M^{-1}} &=& \frac{\sigma}{\Delta\theta^2}\sum_k M_{k+\frac12}\left(\frac{u_{k+1}}{M_{k+1}}-\frac{u_{k}}{M_{k}}\right) \left(\frac{v_{k}}{M_{k}}-\frac{v_{k+1}}{M_{k+1}}\right) \\
			&=& \langle u,Q_N(v)\rangle_{M^{-1}}. \nonumber
	\end{eqnarray*}
	From \eqref{eq:Q_N}, we also deduce \eqref{eq:decay_L2_discrete}.

	Moreover, using once again the Abel formula:
	\begin{eqnarray*}
		\langle Q_N(u),\ln \frac{u}{M}\rangle &=& \frac{\sigma}{\Delta\theta^2}\sum_k M_{k+\frac12}\left(\frac{u_{k+1}}{M_{k+1}}-\frac{u_{k}}{M_{k}}\right) \left(\ln \frac{u_{k}}{M_{k}}-\ln \frac{u_{k+1}}{M_{k+1}}\right).
	\end{eqnarray*}
	Thus, denoting $x=\frac{u_{k+1}}{M_{k+1}}$ and $y=\frac{u_{k}}{M_{k}}$, we have an expression of the form:
	\begin{displaymath}
		(x-y)(\ln y - \ln x) = (x-y) \ln \frac{y}{x} \leq 0
	\end{displaymath}
	for any $x,y>0$. We deduce \eqref{eq:decay_entropy_discrete}.
\end{proof}

\subsubsection{Explicit Euler}

The Euler method can be used to discretize in time the collisional part of the kinetic equation \eqref{eq:collision}:
\begin{equation*}
	f^{n+1} = f^n + \Delta t Q_N(f^n) = (\text{Id} + \Delta t Q_N) f^n.
\end{equation*}
A sufficient condition for the $L^\infty$ stability of the scheme is to have the matrix $\text{Id} + \Delta t Q_N$ positive (i.e. all coefficients positive). This sufficient condition leads to the following CFL condition:
\begin{equation}
	\label{eq:CFL_explicit_Euler}
	\max_k \{|b_k|\}\, \frac{\sigma\Delta t}{\Delta\theta^2}<1,
\end{equation}
which is usual for diffusion type operator. Moreover, if the CFL condition is met, then positivity and mass are preserved.
\begin{remark}
	We can find an explicit sufficient condition to guarantee the CFL condition \eqref{eq:CFL_explicit_Euler}. Indeed, writing:
	\begin{eqnarray*}
		\frac{M_{k+\frac12}}{M_k} &=& \frac{\exp \left(\frac{\mu_f}{\sigma}\cos( \theta_{k+\frac12}-\overline{\theta}) \right)}{\exp \left(\frac{\mu_f}{\sigma}\cos( \theta_{k}-\overline{\theta})\right)} = \exp \left(\frac{\mu_f}{\sigma}[\cos( \theta_{k+\frac12}-\overline{\theta}) - \cos( \theta_{k}-\overline{\theta})] \right) \\
			      & = &\exp \left(-2\frac{\mu_f}{\sigma}\,\sin\left(\theta_k-\overline{\theta}+\frac{\Delta\theta}{4}\right) \sin\left(\frac{\Delta\theta}{4}\right)\right),
	\end{eqnarray*}
	where we have used the identity $\cos \alpha-\cos \beta = -2 \sin \frac{\alpha+\beta}{2}\,\sin \frac{\alpha-\beta}{2}$. We deduce
	\begin{displaymath}
		\frac{M_{k+\frac12}}{M_k} \leq \exp\left(2\frac{\mu_f}{\sigma}\sin\left(\frac{\Delta\theta}{4}\right)\right),
	\end{displaymath}
	and find:
	\begin{displaymath}
		\max|b_k| \leq 2 \exp\left(2\frac{\mu_f}{\sigma}\sin\left(\frac{\Delta\theta}{4}\right)\right) = 2 + \frac{\mu_f}{\sigma}\Delta\theta + \mathcal{O}(\Delta\theta^2).
	\end{displaymath}
	This leads to the tractable (sufficient) CFL condition: 
	\begin{equation}\label{eq:CFL-collision-tractable}
		\frac{2\sigma\Delta t}{\Delta\theta^2}< \exp\left(-2 \frac{\mu_f}{\sigma}\sin\left(\frac{\Delta\theta}{4}\right)\right).
	\end{equation}
\end{remark}

\begin{algorithm}
	\caption{Collision part eq. \eqref{eq:collision}}
	\label{alg:collision}
	\begin{algorithmic}[1]
		\Procedure{Collision}{$f(\theta_k),\Delta t$}
		\State ${\bf j} = \sum_k \omega_k f_k \Delta\theta;$  $\;\;\overline{\theta} = angle({\bf j})$
		\State $M_k = \exp(\frac{\mu_f}{\sigma} \cos(\theta_k-\overline{\theta}))$
		\For {$k$ in $1:N$}
		\State $Q_N(f)_k = \frac{\sigma}{\Delta\theta^2} \cdot(a_k f_{k-1} - b_k f_k + c_k f_{k+1})$
		\EndFor
		\State $f \;\;+\!\!=\; \Delta t \cdot Q_N(f)$
		\State Return $f$
		\EndProcedure
	\end{algorithmic}
\end{algorithm}

\subsubsection{Adaptative time step for the collision}

One of the difficulties in computing an approximate solution to \eqref{eq:kinetic_Amic} is coping with the associated CFL condition \eqref{eq:CFL_explicit_Euler}. Indeed, the existence of a locally high $|j(f)|$ greatly decreases the right-hand side of \eqref{eq:CFL_explicit_Euler}, which penalizes the whole algorithm. Hence using a global CFL condition for the transport part and the collision part can lead to extremely long computation time. 

We propose to decouple the time steps for the transport part and the collision equation at each time step, by using an adaptive method for the latter. Technically, we use the maximal time step associated with the CFL condition to  solve the  transport part \eqref{eq:transport}, $ \Delta t=\Delta x$, which incidentally has the advantage of minimizing numerical diffusion. Then, for each $ (x_i, y_j) $, we consider  \eqref{eq:collision} as a differential equation with final time $ \Delta t$, which we solve by using the method described in Section \ref{sec:collision} with a variable time step $\delta t$ that needs to be recomputed at each time $0\leq s \leq \Delta t$: 
\begin{equation*}
	\delta t(s):=\min\left(\frac{\Delta \theta^2}{2\sigma}\exp\left(\mu_f\frac{-2\sin(\frac{\Delta\theta}{4})}{\sigma}\right), \Delta t -s\right).
\end{equation*}

This method also works for a constant relaxation $\mu(f)=\mu_0$, and can be preferred because it minimizes the numerical diffusion in the transport equation. Particularly when the constant $\sigma=\frac{\mu}{D}$ is large, in which case the collision CFL \eqref{eq:CFL-collision-tractable} is much smaller than the transport CFL \eqref{eq:CFL-transport}.

A comparison between the errors done by the standard and adaptive schemes, respectively, is shown in Figure \ref{fig:collision-error}.

\begin{figure}
	\centering
	\begin{subfigure}[t]{60mm}
		\centering
		\includegraphics{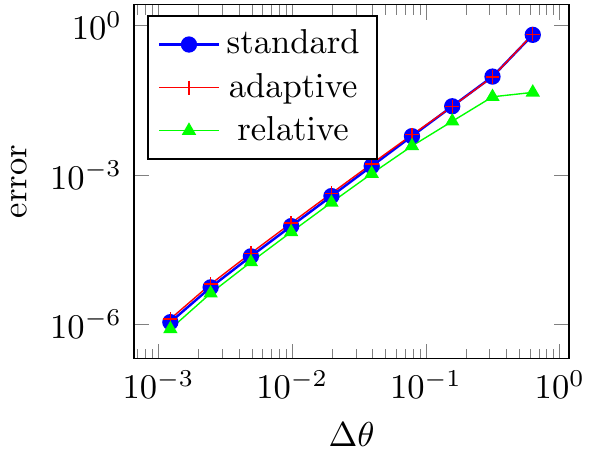}
		\caption{Error in the standard and adaptive models. The {\it standard} and {\it adaptive} curves show the maximum of the difference between the computed distributions and the one with lowest $\Delta \theta$. The "relative" curve shows the maximum of the difference between the standard and adaptive solutions.}
		\label{fig:collision_error_graph}
	\end{subfigure}
	\begin{subfigure}[t]{60mm}
		\centering
		\includegraphics{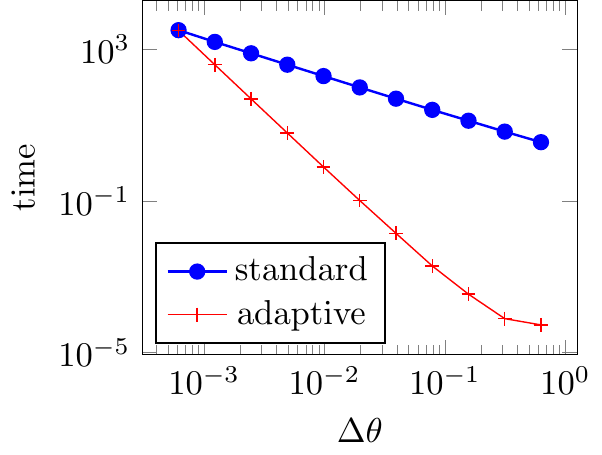}
		\caption{Comparison of the time (in seconds) needed to compute the solution. Notice that the two curves converge to the leftmost point, which was taken as a reference. Understandably, the computational advantage of the adaptive method increases with $\Delta \theta$.}\label{fig:collision_time_graph}
	\end{subfigure}

	\caption{Comparison of the standard and adaptive methods for the collision operator (Vicsek). Parameters are: $\mu=1.0$, $\sigma=0.2$, $\rho=1.0$, $\Delta t=8.458\cdot 10^{-7}$ (standard), $\Delta t=0.1$ (adaptive), $T=1.0$. The initial condition is $f_0(\theta)=\rho\left(1+\frac{1}{5}\underset{k=1}{\overset{5}{\sum}}\cos(p_k\theta)\right) $ where $p_1=1$ and $p_k $ is the prime following $p_{k-1}$.}\label{fig:collision-error}
\end{figure}

\subsection{Numerical scheme for the transport operator}

We use an upwind finite-difference method to solve the transport equation \eqref{eq:transport}. More accurate schemes are available (such as WENO scheme \cite{shu_high_2009}), however stability and positivity preserving are more essential in our study since we investigate the long-time behavior of the dynamics. We fix $M>0$ and consider a uniform discretization of the interval $ [0,L) $ in $M$ points with $x_i=i\Delta x$, $y_j=j\Delta y$, and $\Delta x=\Delta y=\frac{L}{M}$. To discretize the kinetic equation, we use:
\begin{equation*}
	\cos \theta \,\partial_x f= \begin{cases} \cos \theta\,\frac{f(x_{i}) -f(x_{i-1})}{\Delta x} + \mathcal O(\Delta x), \quad \text{if } \cos \theta\geq0 \\
		\cos \theta\,\frac{f(x_{i+1}) -f(x_{i})}{\Delta x} + \mathcal O(\Delta x), \quad \text{if } \cos \theta \leq 0,
	\end{cases}
\end{equation*}
and similarly for $\sin(\theta)\partial_y f$.
Using this discretization, the standard Euler scheme gives as CFL condition:
\begin{equation}
	\label{eq:CFL-transport}
	c\frac{\Delta t}{\Delta x} < 1.
\end{equation}

\begin{algorithm}
	\caption{Transport part eq. \eqref{eq:transport}}
	\label{alg:transport}
	\begin{algorithmic}[1]
		\Procedure{Transport}{$f(x_i,y_j,\theta_k),\Delta t$}
		\For {$i,j,k$}
		\State $F_{i+\frac12,j,k} = \left\{
			\begin{array}{ll}
				c\,\cos \theta_{i+\frac12}\,f_{i,j,k}& \text{if }\cos \theta_{i+\frac12}\leq0  \\
				c\,\cos \theta_{i+\frac12}\,f_{i+1,j,k}& \text{if }\cos \theta_{i+\frac12}\geq0  \\
			\end{array}
		\right.$
		\EndFor
		\For {$i,j,k$}
		\State $f_{i,j,k} \;\;+\!\!=\; -\frac{\Delta t}{\Delta x} \cdot(F_{i+\frac12,j,k}-F_{i-\frac12,j,k})$
		\EndFor
		\For {$i,j,k$}
		\State $F_{i,j+\frac12,k} = \left\{
			\begin{array}{ll}
				c\,\sin \theta_{j+\frac12}\,f_{i,j,k} & \text{if }\sin \theta_{i+\frac12}\leq0  \\
				c\,\sin \theta_{j+\frac12}\,f_{i,j+1,k}  & \text{if }\sin \theta_{i+\frac12} \geq0
			\end{array}
		\right.$
		\EndFor
		\For {$i,j,k$}
		\State $f_{i,j,k} \;\;+\!\!=\; -\frac{\Delta t}{\Delta y} \cdot(F_{i,j+\frac12,k}-F_{i,j-\frac12,k})$
		\EndFor
		\State Return $f$
		\EndProcedure
	\end{algorithmic}
\end{algorithm}

\subsection{Summary full scheme}

The full algorithm is finally a splitting between the transport and collision part. Notice that the time step $\Delta t$ should satisfy both CFL conditions \eqref{eq:CFL_explicit_Euler} and \eqref{eq:CFL-transport}. In general, the collisional CFL \eqref{eq:CFL_explicit_Euler} is more restrictive. Therefore, the transport equation will be solved with a small CFL corresponding to large numerical viscosity. Since we aim at studying the large time behavior of the dynamics, this numerical viscosity might drastically change the outcome. Thus, we propose to use an adaptive method for the collisional operator. The idea is to simply iterate $K$ `small' steps $\delta t=\Delta t/K$ to update the collision part choosing $K$ such that $\delta t$ satisfies the CFL condition  \eqref{eq:CFL-collision-tractable}. 


\begin{algorithm}
	\caption{Collision part eq. \eqref{eq:collision}}
	\label{alg:collision_adapt}
	\begin{algorithmic}[1]
		\Procedure{CollisionAdapt}{$f(\theta_k),\Delta t$}
		\State Find $K$ such that $\delta t=\Delta t/K$ satisfies \eqref{eq:CFL-collision-tractable}
		\For {$s$ in $1:K$}
		\State $f = \text{\bf Collision}(f,\delta t)$
		\EndFor
		\State Return $f$
		\EndProcedure
	\end{algorithmic}
\end{algorithm}

\begin{algorithm}
	\caption{Full kinetic eq. \eqref{eq:kinetic_Vicsek}}
	\label{alg:full_kinetic}
	\begin{algorithmic}[1]
		\State Fix $\Delta t<\min(\Delta x,\Delta y)/c$
		\State $t=0$
		\While {$t<T$}
		\State $f^* = \text{\bf Transport}(f^n,\Delta t)$
		\For {$i,j$}
		\State $f_{i,j,k}^{n+1} = \text{\bf CollisionAdapt}(f_{i,j,k}^*,\Delta t)$
		\EndFor
		\State $t \;+\!\!= \Delta t$ 
		\EndWhile
		\State Return $f$
	\end{algorithmic}
\end{algorithm}

\section{Numerical experiments}
\setcounter{equation}{0}

\subsection{Homogeneous case}

To first investigate our numerical scheme, we study the homogeneous equation, thus solving only the collision operator \eqref{eq:collision}. We present our numerical experiments with the Vicsek model \eqref{eq:kinetic_Vicsek}, however results are similar with the DFL dynamics except that the time step $\Delta t$ may have to be adapted (since the CFL condition depends on $|{\bf j}|$ which varies over time).

As a first sanity check, we estimate the accuracy of the scheme. With this aim, we fix a final time $T=1$ and time step $\Delta t=.001$. Then, we vary the meshgrid in $\theta$, taking $\Delta\theta\in\{\frac{2\pi}{8},\frac{2\pi}{16},\dots,\frac{2\pi}{128}\}$ and estimate the $L^2$ error with the reference solution $f_{ref}$ computing with $\Delta\theta=\frac{2\pi}{256}$. For the initial condition, we use a smooth  initial condition:
\begin{equation}
	\label{eq:f_0}
	f_0(\theta) = (1.1+\cos 4\theta)\cdot\exp\left(-\cos \big(\pi (s+s^8)\big)\right),\qquad \text{with } s=\theta/2\pi.
\end{equation}
We use a rather complicated expression to make sure that $f_0$ is non-symmetric. When $f_0$ is symmetric, the mean direction ${\bar \theta}$ is preserved over time, thus the Vicsek dynamics \eqref{eq:homog} becomes a linear evolution equation. Twisting the initial condition $f_0$ guarantees to have a fully non-linear equation.

In Figure \ref{fig:beauty_f0_t1}-left, we plot the initial condition $f_0$ along with the reference solution $f_{ref}$ at $t=1$. The $L^2$ error for various discretizations is given in log scale in Figure  \ref{fig:beauty_f0_t1}-right. We observe that the error is decaying quadratically as expected.

Moreover, we also investigate the large-time behavior of the solution. First, we measure the evolution of the free entropy $\mathcal{F}$ over time and we observe that it is strictly decreasing (fig.~\ref{fig:cv_free_energy}-left). Second, we estimate the rate of convergence of $f(t)$ toward an  equilibrium distribution. Using semi-log scale in fig.~\ref{fig:cv_free_energy}-right, we observe a linear decay indicating that the convergence is exponential.

\begin{figure}[ht]
	\centering
	\includegraphics[width=.47\textwidth]{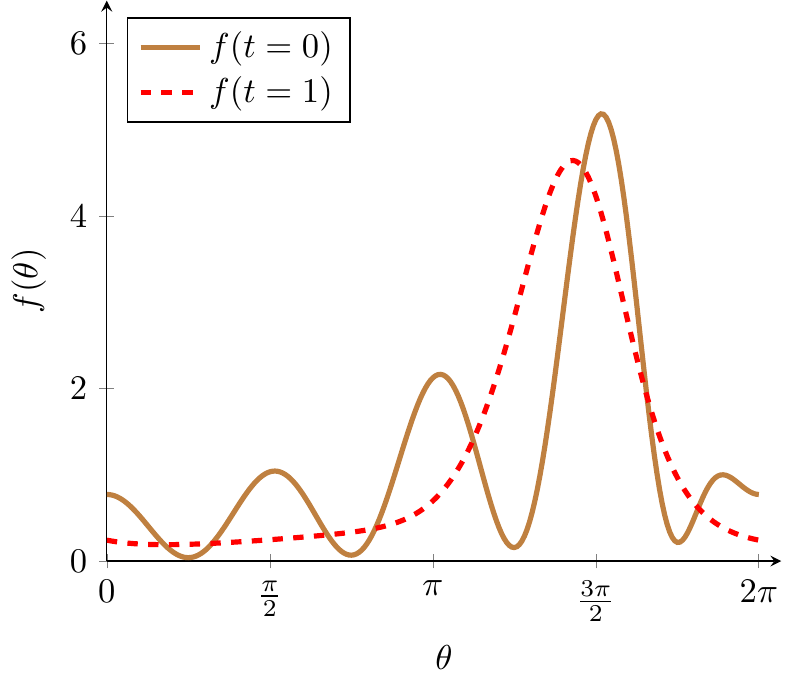} \quad
	\includegraphics[width=.47\textwidth]{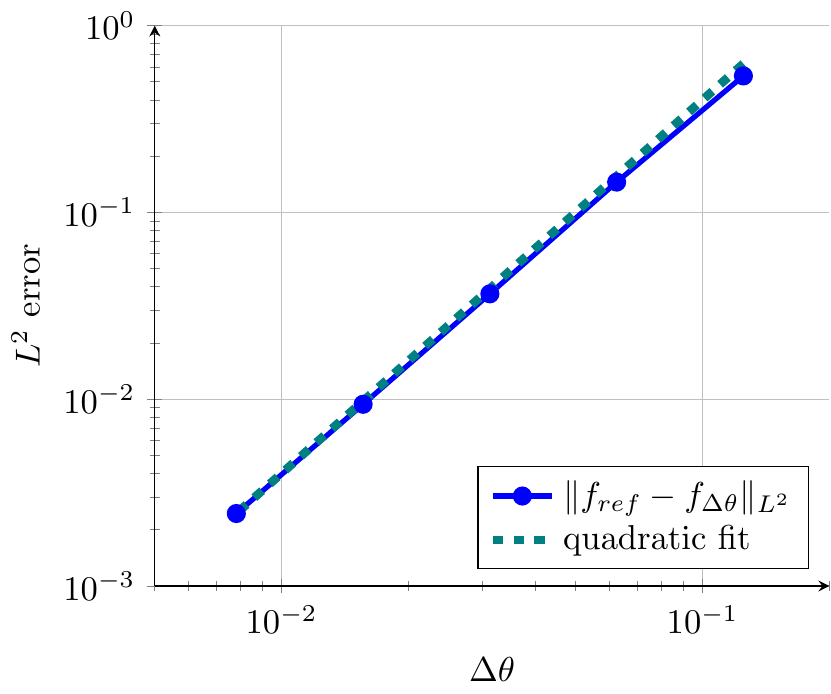}
	\caption{{\bf Left:} the initial condition $f_0$ \eqref{eq:f_0} and the reference solution $f_{ref}$ (dashed) computed after $t=1$ (with $\Delta t=10^{-3}$ and $\Delta\theta=\frac{2\pi}{256}$). {\bf Right:} $L^2$ error in log scale between the solution $f_{\Delta\theta}$ with the reference solution $f_*$ at $t=1$. We observe a quadratic accuracy.}
	\label{fig:beauty_f0_t1}
\end{figure}

\begin{figure}[ht]
	\centering
	\includegraphics[width=.47\textwidth]{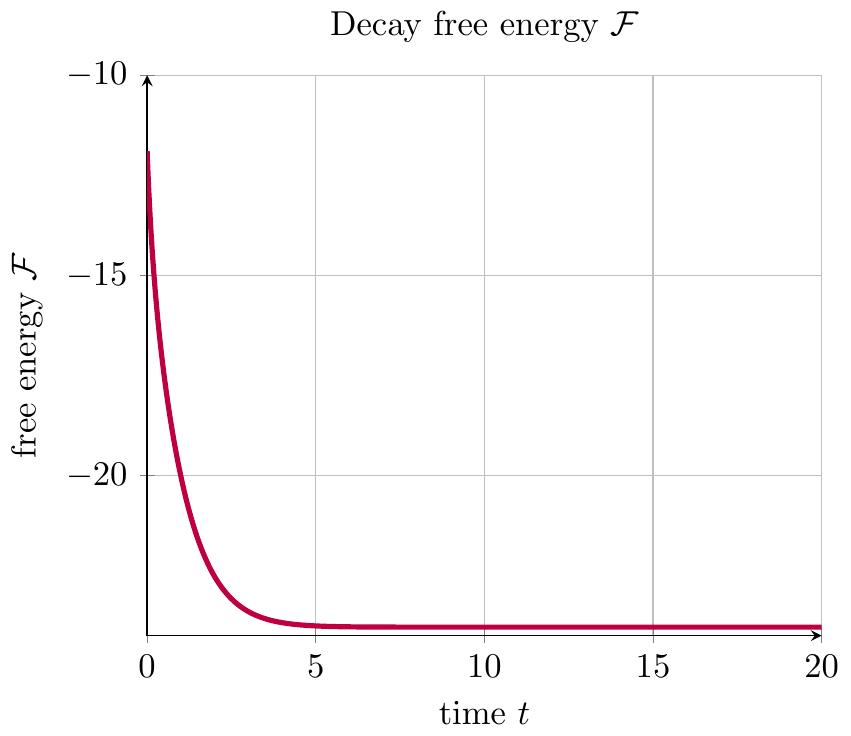} \quad
	\includegraphics[width=.47\textwidth]{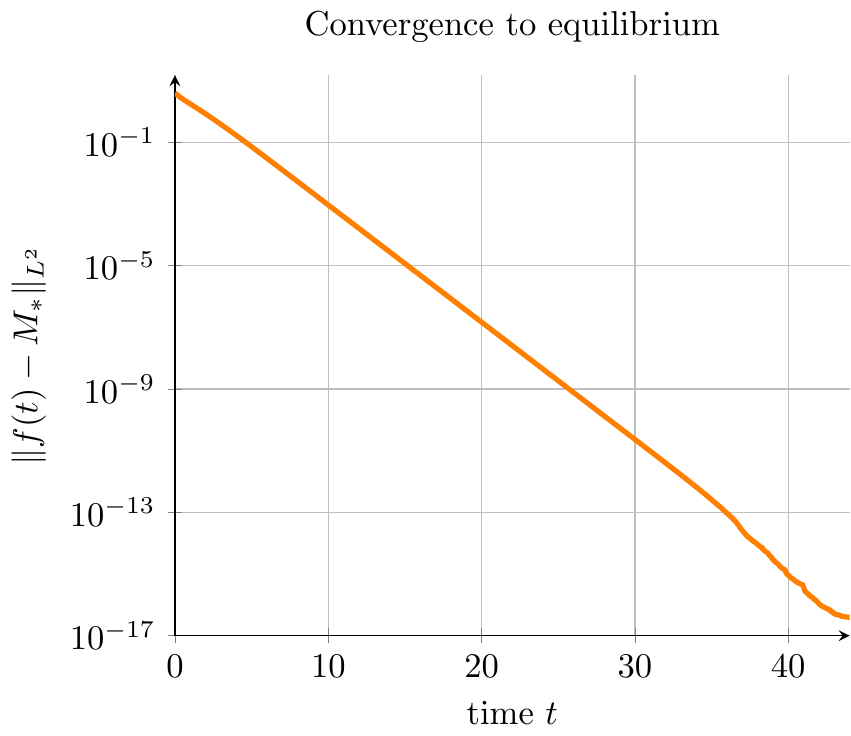}
	\caption{{\bf Left:} evolution of the free energy $\mathcal{F}$ \eqref{eq:free_energy} over time. The function is strictly decreasing. {\bf Right:} $L^2$ error between the solution $f(t)$ and its equilibrium distribution $M_*$. Since it uses semi-log scale, the convergence is actually exponential.}
	\label{fig:cv_free_energy}
\end{figure}


\subsection{Band formation}

\begin{figure}
	\centering
	\includegraphics{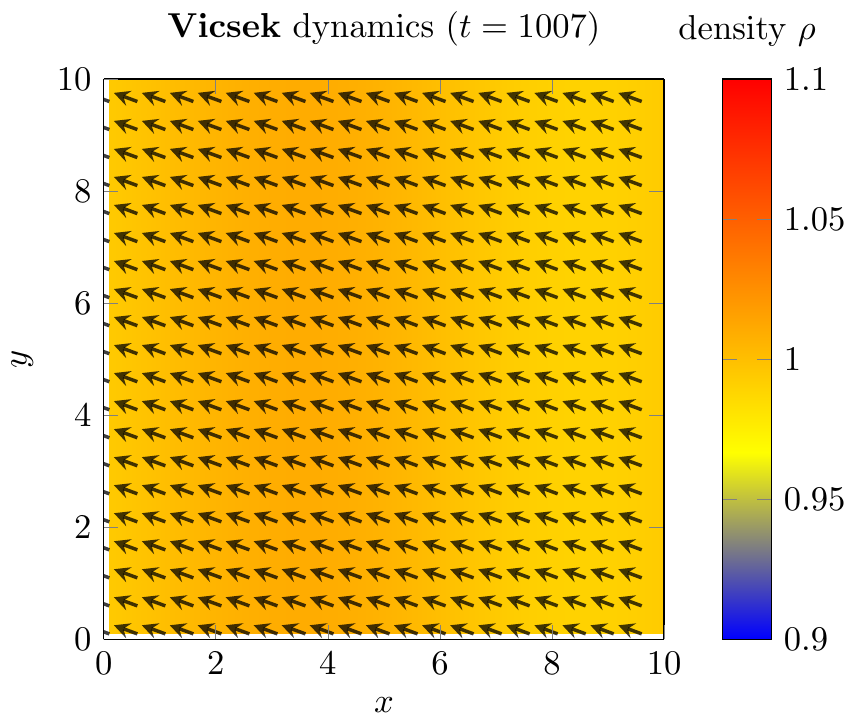}
	\caption{Typical shape of the long-time solution observed in the case of the Vicsek interaction model (obtained starting from a random initial condition) at $t=1000$. Parameters are: $\mu=1.0$, $\sigma=0.2$, $c=1.0$, $L=10.0$, $\Delta x=\Delta y=0.1$, $\Delta \theta = \frac{2\pi}{30}$.}
	\label{fig:Vicsek_longtime}
\end{figure}

\begin{figure}
	\centering
	\begin{minipage}[t]{60mm}
		\centering
		\includegraphics{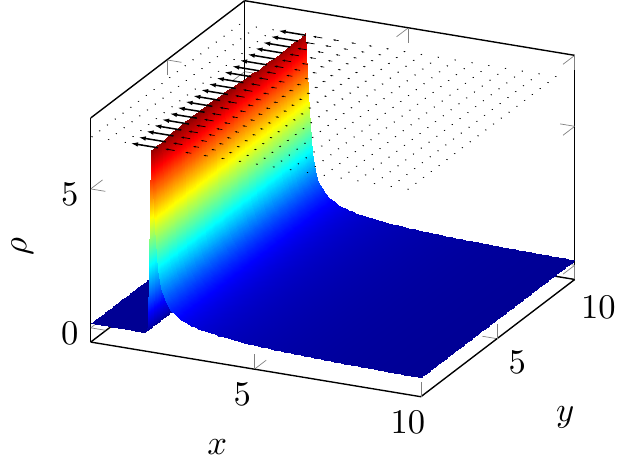}
		\subcaption{Starting from a random initial condition $t=1000.0$ (${\bar \rho}\approx 0.0766$).}\label{fig:band-fromrand-1000}
	\end{minipage}
	\begin{minipage}[t]{60mm}
		\centering
		\includegraphics{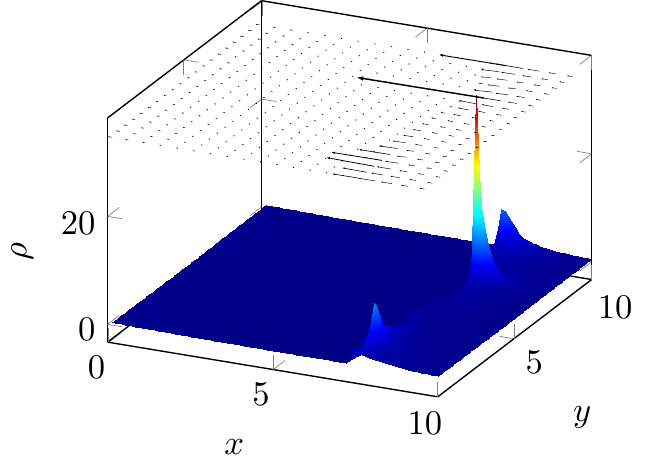}
		\subcaption{Starting from a random initial condition $t=1227.0$.}\label{fig:band-fromrand-1227}
	\end{minipage}

	\begin{minipage}[t]{60mm}
		\centering
		\includegraphics{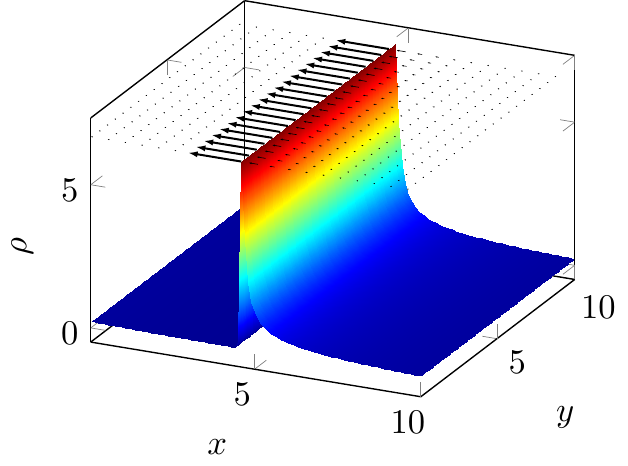}
		\subcaption{Starting from a homogeneous initial condition, $t=1007.0$ (${\bar \rho}=0.0763$).}\label{fig:band-homog}
	\end{minipage}
	\begin{minipage}[t]{60mm}
		\centering
		\includegraphics{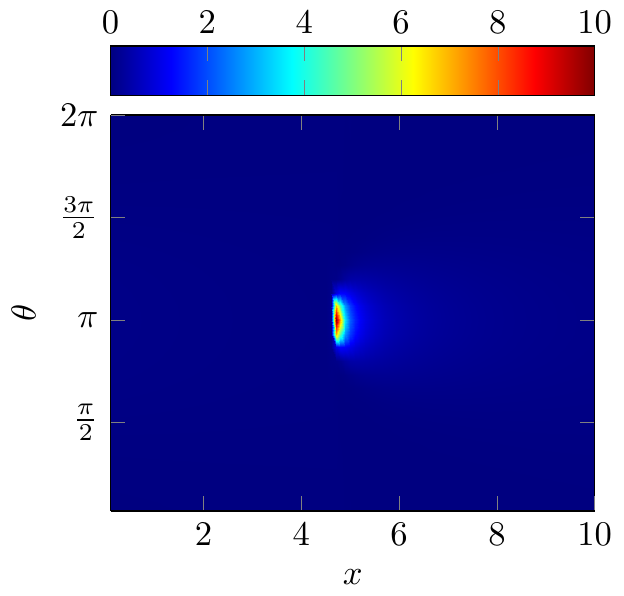}
		\subcaption{Pseudo-1D code, $t=1005.0$ (${\bar \rho}=0.0763$). Here the real value of $f$ is represented as a function of $x$ and $\theta$.}\label{fig:band-pseudo1D}
	\end{minipage}

	\begin{minipage}[t]{60mm}
		\centering
		\includegraphics{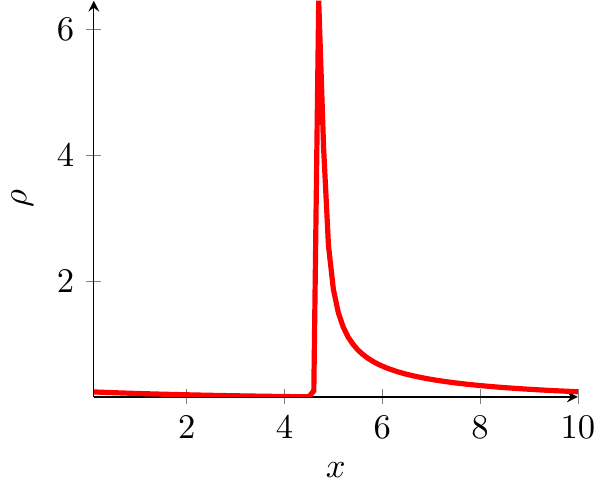}
		\subcaption{Pseudo-1D code, $t=1005.0$ (${\bar \rho}=0.0763$). $\rho$ is represented as a function of $x$.}\label{fig:band_pseudo1D_rho}
	\end{minipage}
	\caption{Different observations of bands. The surface plot corresponds to the observed density $\rho=\int_{\mathbb S^1}f(t, {\bf x}, \omega)\dd\omega$. The arrows on the top correspond to the local flux ${\bf j}({\bf x})$. Figures \ref{fig:band-fromrand-1000} and \ref{fig:band-fromrand-1227} were obtained starting from the same random initial condition. Figure \ref{fig:band-homog} was obtained starting from a homogeneous in $y$ initial condition. Figure \ref{fig:band-pseudo1D} was obtained by a one-dimensional version of the code.  In all cases, the same set of parameters was used: $\mu=1.0$, $\sigma=0.2$, $c=1.0$, $L=10.0$, $\Delta x=\Delta y=0.1$, $\Delta \theta = \frac{2\pi}{30}$.}\label{fig:bands}
\end{figure}

In the Vicsek model \eqref{eq:kinetic_Vicsek}, we did not observe the formation of any bands. Rather, the dynamics always converge to a robust global alignment dynamics, where the spatial distribution (first moment of $f$) converges to a constant. The typical long-time behavior is represented in Figure \ref{fig:Vicsek_longtime}. We postulate that the long-time behavior of this equation is just to converge to a uniform distribution of Von Mises equilibria to the homogeneous equation.

On the other hand, the DFL model \eqref{eq:kinetic_Amic} of interaction leads to the observation of bands. Typically, for a fixed set of parameters, we observed two different scenarios regarding the behavior of the local density $\rho$ and mean value ${\rho}$
\begin{equation}
	\label{eq:density_f_bar}
	\rho(t, {\bf x}) = \int_{\mathbb S^1}f(t, {\bf x}, \omega)\dd\omega\, ,\qquad {\bar \rho} =\frac{\int_{[0,L]^2\times \mathbb S^1} f({\bf x}, \omega)\dd{\bf x}\dd \omega}{2\pi L^2 } .
\end{equation}
For a fixed mean value ${\rho}$, when the strength of interaction $\mu$ is small compared to the diffusion parameter $\sigma$ (i.e. $\mu\gg\sigma$), we observe that the solution converges to a uniform steady state. However, when $\mu\ll\sigma$, we observe the formation of bands as shown in Figure \ref{fig:bands} (in which the x-axis has been reversed to provide better aesthetics). Thus, we retrieve an equivalent of the phase transition dynamics noticed by Frouvelle and Liu in \cite{frouvelle_dynamics_2012}. Those bands were first noticed starting from a random initial condition (Figure \ref{fig:band-fromrand-1000}). Even though they literally emerge from chaos, they appear to be only meta-stable, as their small inhomogeneity in the direction perpendicular to the propagation amplifies slowly by attracting the neighbor particles and finally lead to high and localized concentrations as shown in Figure \ref{fig:band-fromrand-1227}. At this point the computation is difficult to continue due to the extremely high computation times required by the CFL condition \eqref{eq:CFL-collision-tractable}. Starting from an initial condition which is homogeneous in one direction (e.g. in $y$), however,  the observed bands are very stable in time and can be kept alive for apparently an arbitrarily long time (the homogeneity being preserved by our scheme). Such a band is represented in Figure \ref{fig:band-homog}. The initial condition we used is the following:
\begin{equation*}
	f_0(x, y, \theta)={\bar \rho}\left(1+\frac{1}{10}\underset{k=1}{\overset{5}{\sum}}\cos(p_k\theta) + \cos\left(2p_k\pi \frac{x}{L} \right)\right).
\end{equation*}
Bands were also observed in a modified one-dimensional  model which we encoded to take advantage of the preservation of homogeneity in one direction. A resulting band is presented in Figure \ref{fig:band-pseudo1D} and \ref{fig:band_pseudo1D_rho}.

\begin{figure}
	\centering
	\includegraphics{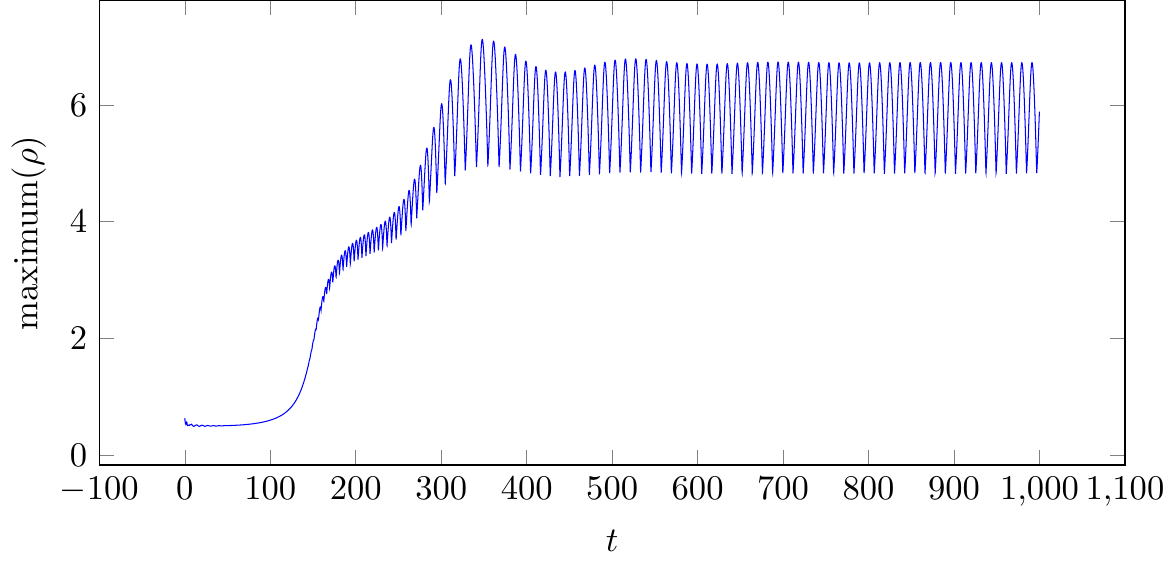}
	\caption{Maximal value of $\rho(t, {\bf x}) $ as a function of $t$.}\label{fig:max-band}
\end{figure}

Numerical evidences show that the bands cannot be understood as traveling wave solutions to the kinetic equation \eqref{eq:kinetic_Vicsek}, as one may believe at first sight. Indeed, there remains an inner motion inside the bands, that we can reveal by monitoring the maximal value of $\rho$ through time (see Figure \ref{fig:max-band}). This reveals an asymptotically periodic behavior that strongly resembles  the notion of pulsating fronts, which has been extensively studied in the context of reaction-diffusion equations \cite{xin_front_2000}. A deeper analytical understanding of this phenomenon is left for future work.

\begin{figure}
	\centering
	\begin{minipage}[t]{60mm}
		\centering
		\includegraphics{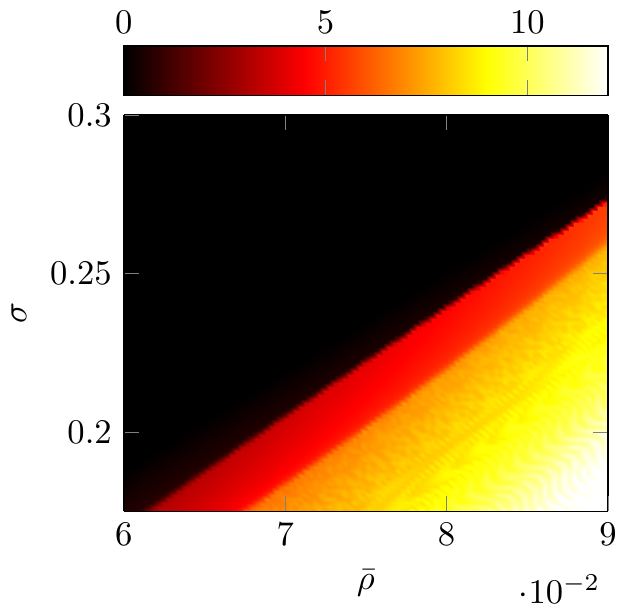}
		\subcaption{Entropy against the uniform distribution}\label{fig:band-entropy-uniform}
	\end{minipage}
	\begin{minipage}[t]{60mm}
		\centering
		\includegraphics{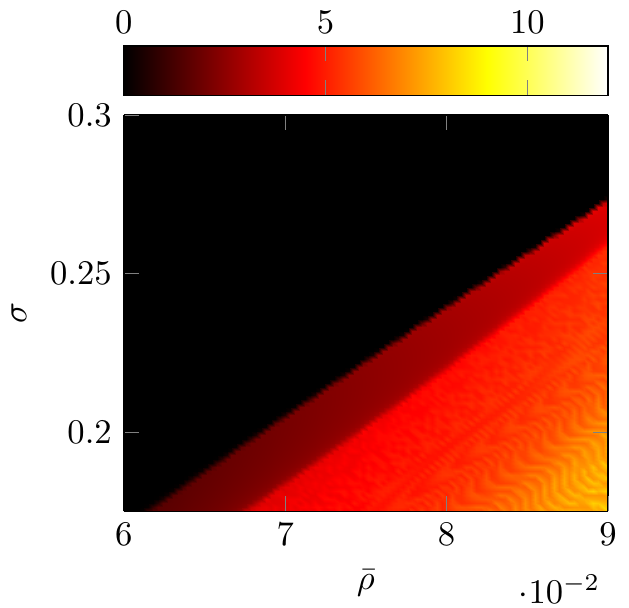}
		\subcaption{Entropy against the Von Mises distribution}\label{fig:band-entropy-vonmises}
	\end{minipage}

	\begin{minipage}[t]{60mm}
		\centering
		\includegraphics{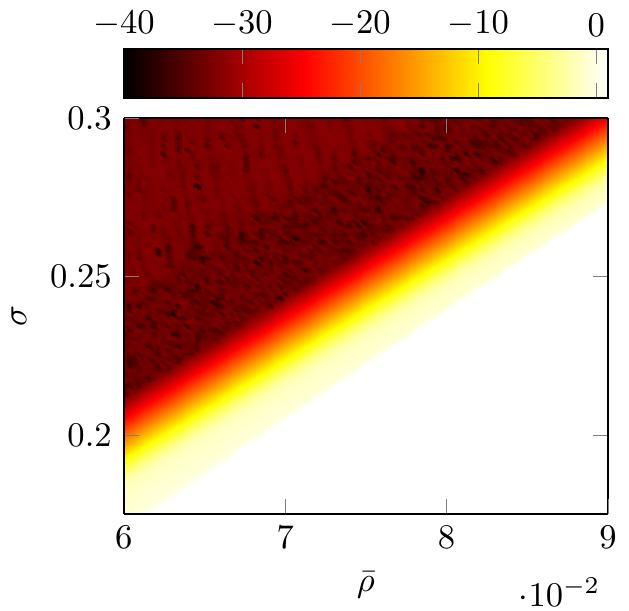}
		\subcaption{Log-entropy against the uniform distribution}\label{fig:band-entropy-log-uniform}
	\end{minipage}
	\begin{minipage}[t]{60mm}
		\centering
		\includegraphics{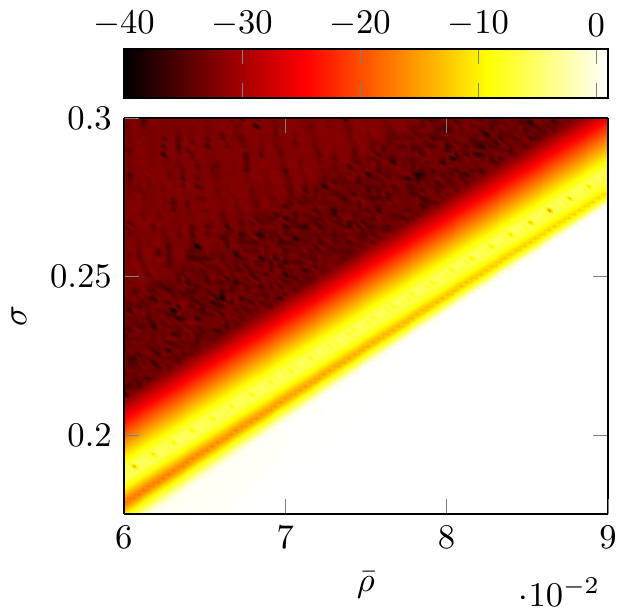}
		\subcaption{Log-entropy against the Von Mises distribution}\label{fig:band-entropy-log-vonmises}
	\end{minipage}

	\caption{Entropy as a function of ${\bar \rho}$ and $\sigma$.}\label{fig:band-entropy}
\end{figure}

Finally, to strengthen the link between the phase transition and the formation of bands, we show in Figure \ref{fig:band-entropy} two kinds of entropy computed for a range of values of the diffusion coefficient $d$ and the mean value of the initial condition ${\bar \rho}$. Figure \ref{fig:band-entropy-uniform} represents the entropy of $f$ computed against the uniform distribution of the same mass:
\begin{equation*}
	\mathcal E_u[f] = \int_0^L\int_0^L\int_0^{2\pi}f(t, {\bf x}, \theta)\log\left(\frac{f(t, {\bf x}, \theta)}{{\bar \rho}}\right)\dd\theta\dd{\bf x}.
\end{equation*}
Figure \ref{fig:band-entropy-vonmises} represents the generalized entropy of $f$ computed against the corresponding Von Mises distribution:
\begin{equation*}
	\mathcal E_{VM}[f] = \int_0^L\int_0^L\int_0^{2\pi}f(t, {\bf x}, \theta)\log\left(\frac{f(t, {\bf x}, \theta)}{M[{\bar \rho}](\theta)}\right)\dd\theta\dd{\bf x},
\end{equation*}
where $M[{\bar \rho}](\theta)=2\pi{\bar \rho} \frac{\exp\left(\frac{\mu\kappa}{d}\cos(\theta)\right)}{\int_0^{2\pi}\exp\left(\frac{\mu\kappa}{d}\cos(\theta)\right)\dd\theta}$ is the only candidate as stationary Von Mises distribution, $\kappa$ satisfying the compatibility condition
\begin{equation*}\label{eq:con-compat}
	2\pi{\bar \rho}\frac{\int_0^{2\pi}\cos\theta\exp\left(\frac{\mu\kappa}{d}\cos(\theta)\right)\dd\theta}{\int_0^{2\pi}\exp\left(\frac{\mu\kappa}{d}\cos(\theta)\right)\dd\theta}=\kappa.
\end{equation*}
Let us recall that the latter has only one solution $\kappa=0$ when $\sigma\geq \pi \mu {\bar \rho} $, and has exactly one positive solution when  $\sigma< \pi \mu {\bar \rho} $ \cite{frouvelle_dynamics_2012}.

The match between the two plots suggests that the latter stationary state candidate is never stable except when $\kappa=0$. Indeed for $\sigma\geq \pi \mu {\bar \rho} $, the uniform distribution is stable for the homogeneous problem and $\kappa=0$; this corresponds to the top-left part of Figure \ref{fig:band-entropy-uniform}, which suggests that this stability is transferred to the inhomogeneous problem \eqref{eq:kinetic_Amic}. For $\sigma< \pi \mu {\bar \rho} $, however, we have $\kappa>0$ and the stable state for the homogeneous problem is described by the corresponding Von Mises distribution $M[{\bar \rho}](\theta)$; Figure \ref{fig:band-entropy-vonmises} suggests that the inhomogeneous problem behaves otherwise, neither the uniform nor the homogeneous Von Mises distribution corresponding to the long-time behavior of the equation, except possibly in a very small area near $ \sigma\approx \pi \mu{\bar \rho}$ (which appears more clearly in the log-plots \ref{fig:band-entropy-log-uniform} and \ref{fig:band-entropy-log-vonmises}). Instead, the instability of both homogeneous stationary states could be at the origin of the formation of bands.

\section{Conclusion}

In this manuscript, we have introduced a numerical scheme to solve kinetic equations related to the Vicsek model. Despite the additional difficulty to treat non-conservative dynamics, the numerical scheme is able to preserve positivity and entropy which allow to study the long-time behavior of the dynamics. Of particular interest in the formation of bands that were first observed at the microscopic model (i.e. particle simulations). The scheme is able to capture such band formation at the kinetic level.

It remains many open questions related to the kinetic equation of the Vicsek model. First, we still do not know if asymptotically there is an analytical description of the band. Having an {\it Ansatz} will help to study the convergence of the dynamics toward a band formation. However, the lack of entropy in the non-homogeneous case makes this task complicated. Another open problem is to investigate how the Vicsek model in a certain regime (i.e. low density, large coefficients) is able to create bands. We do not observe such band formation for the original Vicsek model at the kinetic level, bands only form with the DFL dynamics. This might indicate that the propagation of chaos, which links particle and kinetic formulations, is not valid in this regime.

\bibliographystyle{plain}
\bibliography{Vicsek_DFL}

\end{document}